\documentclass[
11pt,%
tightenlines,%
twoside,%
onecolumn,%
nofloats,%
nobibnotes,%
nofootinbib,%
superscriptaddress,%
noshowpacs,%
centertags]%
{revtex4}
\usepackage{ljm}
\usepackage{float}
\usepackage{booktabs}

\newtheorem{remark}{Remark}
\begin{document}


\title{A note on convergence in mean for $d$-dimensional arrays of random vectors in Hilbert spaces under the Ces\`{a}ro uniform integrability}

\author{\firstname{Dat}~\surname{Thai Van}}
\email[E-mail: ]{thaidatpbc@gmail.com}
\affiliation{Department of Mathematics, Vinh University, Nghe An, Vietnam}



\begin{abstract}
    This note establishes convergence in mean of order $p$, $0<p\le 1$ for $d$-dimensional arrays of random vectors in Hilbert spaces under the Ces\`{a}ro uniform integrability conditions.
    In the case where $0<p<1$, our $L_p$ convergence is valid irrespective of any dependence structure.
    In the case where $p=1$, the underlying random vectors are supposed to be pairwise independent.
    The mean convergence results are established for maximal partial sums while previous contributions were so far considered partial sums only. 
    Some results in the literature are extended. Various properties of the Ces\`{a}ro uniform integrability
    of $d$-dimensional arrays of random vectors such as the classical equivalent criterion and
    the de La Vall\'{e}e Poussin theorem are also detailed.
\end{abstract}



\maketitle

    \section{Introduction}
    Throughout, all random vectors are defined on a probability space $\left( \Omega,\mathcal{F}, \mathbb{P} \right)$ and take values in a real separable Hilbert space $\mathcal{H}$ with inner
    product $\langle \cdot , \cdot \rangle$ and corresponding norm $\|\cdot \|$.

    Let $\mathbb{Z}^d_+$, where $d$ is a positive integer, denote the positive integer $d$-dimensional lattice points. The notation $\mathbf {m \prec n}$ (or $\mathbf{n}\succ\mathbf{m}$),
    where $\mathbf m$ = $(m_1, m_2,..., m_d)$ and $\mathbf n$ = $(n_1, n_2,..., n_d) \in \mathbb{Z}^d_+,$
    means that $m_i \leq n_i, 1 \leq i \leq d$. For $\mathbf n$ = $(n_1, n_2,..., n_d) \in \mathbb{Z}^d_+,$ we denote $\mathbf{|{\mathbf{n}}|} = \prod _{i = 1}^d {n_i}$.
    We also use notation $\mathbf{1} = (1,\ldots,1)$.

    The concept of uniform integrability in the Ces\`{a}ro sense for sequences of
    real valued-random variables was introduced by Chandra \cite{chandra1989uniform}. We say that a $d$-dimensional array of real-valued random variables
     $ \left \{ X_{\mathbf{n}}, \mathbf{n} \in \mathbb{Z}^{d}_{+} \right \}$ is said to be \textit{uniformly intergrable in the Ces\`{a}ro sense} if
    \[
    \lim_{a\to \infty}  \sup\limits_{\mathbf{n} \succ \mathbf{1}} \dfrac{1}{\left | \mathbf{n} \right |}  \sum\limits_{\mathbf{i \prec n}}
     \mathbb{E}\left(|X_{\mathbf{i}}|\mathbf{1}\left(|X_{\mathbf{i}}| > a \right) \right) = 0.
    \]

Convergence in mean for sequences and arrays of random variables and
random vectors under uniform integrability conditions was studied by many authors.
We refer to \cite{adler1997mean,cabrera2005mean,chandra1989uniform,hong1999marcinkiewicz-type,thanh2005lp} and the references therein.
However, to our best knowledge, non of these authors considered
convergence in mean for maximal partial normed sums.
This problem was studied by Rosalsky and Th\`{a}nh \cite[Theorems 3.1 and 3.2]{rosalsky2007almost} but under a Kolmogorov-type condition.

In this paper, we establish convergence in mean of order $p$, $0<p\le 1$, for maximal normed partial sums from
$d$-dimensional arrays of $\mathcal{H}$-valued random vectors under the Ces\`{a}ro uniform integrablity conditions.
Some results in Th\`{a}nh \cite{thanh2005lp}, Hong and Volodin
 \cite{hong1999marcinkiewicz-type} are obtained as special cases.

\section{Main results}
In this section, we extend the result of Th\`{a}nh \cite[Theorem 2.1]{thanh2005lp} for $d$-dimensional arrays of Hilbert space-valued
random vectors under the Ces\`{a}ro uniform integrability conditions. The first theorem establishes a result on convergence in mean of order $p$
for maximal partial normed sums
for the case $0<p<1$. In Theorem \ref{theorem21}, we do not require any dependence structure of the underlying random vectors.


        \begin{theorem}\label{theorem21}
    Let $0<p<1$ and let $\{ X_{\mathbf{n}}, \mathbf{n} \in \mathbb{Z}^{d}_{+}\}$ be a $d$-dimensional array of $\mathcal{H}$-valued random vectors.
If $ \left \{ \|X_{\mathbf{n}}\|^p, \mathbf{n} \in \mathbb{Z}^{d}_{+} \right \} $ is uniformly intergrable in the Ces\`{a}ro sense, then
            \begin{equation}\label{21}
                \dfrac{\max_{\mathbf{1\prec k\prec n}} \left\|\sum_{\mathbf{i\prec k}} X_{\mathbf{i}}\right\|}{|\mathbf{n}|^{1/p}} \rightarrow 0 \text{ in } L_p \text{ as } |\mathbf{n}| \rightarrow \infty.
            \end{equation}
        \end{theorem}

        \begin{proof}
Let $\varepsilon > 0$ be arbitrary. Since $\left \{\|X_{\mathbf{n}}\|^p, \mathbf{n} \in \mathbb{Z}^{d}_{+} \right \} $ is uniformly intergrable in the Ces\`{a}ro sense, there exists a constant
$a > 0$ such that
            \begin{equation}\label{22}
                \sup\limits_{\mathbf{n} \succ \mathbf{1}} \dfrac{1}{\left | \mathbf{n} \right |}  \sum\limits_{\mathbf{i \prec n}} \mathbb{E}\left( \|X_{\mathbf{i}}\|^p\mathbf{1}\left( \|X_{\mathbf{i}}\| > a \right) \right) < \varepsilon.
            \end{equation}
For each $\mathbf{i} \succ \mathbf{1}$, put
           $
            Y_\mathbf{i} = X_{\mathbf{i}} \mathbf{1}\left( \|X_{\mathbf{i}}\| > a \right)$  and $ Z_\mathbf{i} = X_{\mathbf{i}}\mathbf{1}\left( \|X_{\mathbf{i}}\| \leq a \right).
            $
By the $c_p$-inequality (see, e.g., \cite[Lemma A.5.1]{gut2013probability}) and \eqref{22}, we have

$$
                               \mathbb{E} \left( \dfrac{\max_{\mathbf{1\prec k\prec n}} \left\|\sum_{\mathbf{i\prec k}} X_{\mathbf{i}}\right\|}{|\mathbf{n}|^{1/p}} \right)^p
                \leq \dfrac{1}{|\mathbf{n}|} \mathbb{E} \left( \max_{\mathbf{1\prec k\prec n}} \left\|\sum_{\mathbf{i\prec k}} Y_{\mathbf{i}}\right\| +
                \max_{\mathbf{1\prec k\prec n}} \left\|\sum_{\mathbf{i\prec k}} Z_{\mathbf{i}}\right\| \right)^p
$$
$$
                \leq \dfrac{1}{|\mathbf{n}|} \mathbb{E} \left( \sum_{\mathbf{i\prec n}} \|Y_\mathbf{i}\| + \sum_{\mathbf{i\prec n}} \|Z_\mathbf{i}\| \right)^p\\
                \leq \dfrac{1}{|\mathbf{n}|} \left( \sum_{\mathbf{i\prec n}} \mathbb{E}\|Y_{\mathbf{i}}\|^p + \mathbb{E}
                \left(  \sum_{\mathbf{i\prec n}} \|Z_\mathbf{i}\| \right)^p \right)
             $$
 \begin{equation}\label{23}
               \leq \dfrac{1}{|\mathbf{n}|}\left( |\mathbf{n}|\varepsilon + \left( |\mathbf{n}| a \right)^p \right)
                = \varepsilon + \dfrac{a^p}{|\mathbf{n}|^{1-p}}.
                     \end{equation}
Since $0<p<1$ and $\varepsilon>0$ is arbitrary, it follows from \eqref{23} that
            \[
            \mathbb{E} \left( \dfrac{\max_{\mathbf{1\prec k\prec n}} \left\|\sum_{\mathbf{i\prec k}} X_{\mathbf{i}}\right\|}{|\mathbf{n}|^{1/p}} \right)^p \rightarrow 0 \text{ as } |\mathbf{n}| \rightarrow \infty.
            \]
Thus \eqref{21} holds.
        \end{proof}

\begin{remark}
    {
        \rm
        In Th\`{a}nh \cite[Theorem 2.1]{thanh2005lp}, the author proved convergence in mean of order $p$ for partial sums
        from $d$-dimensional arrays of
        real-valued random variables under a stronger condition that $\{|X_{\mathbf{n}}|^p, \mathbf{n} \in \mathbb{Z}^{d}_{+}\}$
        is uniformly integrable. It was shown by Gut \cite{gut1992complete} that the concept of uniform integrability in the Ces\`{a}ro sense is strictly weaker
        the concept of uniform integrability. Relationships between uniform integrability and stochastic domination was studied by Rosalsky and Th\`{a}nh \cite{rosalsky2021note}.
    }
\end{remark}

Before establishing result on $L_1$ convergence, we present the following simple lemma. This lemma is $d$-dimensional version of Lemma 3.2 of \cite{moricz1994strong} (see also \cite[Corollary 6]{moricz1977moment}).
        \begin{lemma}\label{lemma24}
            Let $n = (n_1,n_2,\ldots,n_d) \in \mathbb{Z}^{d}_{+}$ and let $\{ X_{\mathbf{n}}, \mathbf{n} \in \mathbb{Z}^{d}_{+}\}$ be a $d$-dimensional array of pairwise independent $\mathcal{H}$-valued random vectors. If  $\mathbb{E} X_\mathbf{i} = 0$ for all $\mathbf{i} \succ \mathbf{1}$, then there exists a constant $C>0$ such that
            \[
                \mathbb{E} \left( \max_{\mathbf{1\prec k\prec n}} \left\| \sum_{\mathbf{i\prec k}} X_{\mathbf{i}} \right\|^2 \right) \leq C\log^2(2n_1)\log^2(2n_2)\cdots \log^2(2n_d) \sum_{\mathbf{i\prec n}} E\|X_{\mathbf{i}}\|^2,
            \]
            with $\log x$ is the natural logaritheorem (base $e$) of $x$, $x\ge 1$.
        \end{lemma}

In the following theorem, we establish a result on $L_1$ convergence.
        When the array $\{ X_{\mathbf{n}}, \mathbf{n} \in \mathbb{Z}^{d}_{+}\}$
is uniformly integrable, this result reduce to the case $p=1$ of Theorem 2.1 of Th\`{a}nh \cite{thanh2005lp}.
We would like to note that Th\`{a}nh \cite{thanh2005lp} established $L_1$ convergence for partial sums of
real-valued random variables while Theorem \ref{theorem22} establishes $L_1$ convergence for maximal partial sums for
random variables taking valued in $\mathcal{H}$. For $L_1$ convergence for partial sums of dependent $\mathcal{H}$-valued random
vectors, we refer the reader to a recent paper by Ord{\'o}{\~n}ez Cabrera et al. \cite{ordonezcabrera2020new}.

        \begin{theorem}\label{theorem22}
            Let $\{ X_{\mathbf{n}}, \mathbf{n} \in \mathbb{Z}^{d}_{+}\}$ be a $d$-dimensional array of
            pairwise independent $\mathcal{H}$-valued random vectors. If $ \left \{ \|X_{\mathbf{n}}\|, \mathbf{n} \in \mathbb{Z}^{d}_{+} \right \}$ is
            uniformly intergrable in the Ces\`{a}ro sense, then
            \begin{equation}\label{24}
                \dfrac{ \max_{\mathbf{1\prec k\prec n}} \left\|\sum_{\mathbf{i\prec k}} (X_{\mathbf{i}} -\mathbb{E}X_{\mathbf{i}})  \right\| }{|\mathbf{n}|} \rightarrow 0 \text{ in } L_1 \text{ as } |\mathbf{n}| \rightarrow \infty.
            \end{equation}
        \end{theorem}
        \begin{proof}
            Let $\varepsilon > 0$ be arbitrary. Since $\left \{\|X_{\mathbf{n}}\|, \mathbf{n} \in \mathbb{Z}^{d}_{+} \right \} $ is uniformly intergrable in the Ces\`{a}ro sense, there exists a
            constant $a > 0$ such that
            \begin{equation}\label{25}
                \sup\limits_{\mathbf{n} \succ \mathbf{1}} \dfrac{1}{\left | \mathbf{n} \right |}  \sum\limits_{\mathbf{i \prec n}} \mathbb{E}\left( \|X_{\mathbf{i}}\|\mathbf{1}\left( \|X_{\mathbf{i}}\| > a \right) \right) < \dfrac{\varepsilon}{2}.
            \end{equation}
            For each $\mathbf{i} \succ \mathbf{1}$, put
            $
            Y_\mathbf{i} = X_{\mathbf{i}} \mathbf{1}\left( \|X_{\mathbf{i}}\| > a \right)$  and $ Z_\mathbf{i} = X_{\mathbf{i}}\mathbf{1}\left( \|X_{\mathbf{i}}\| \leq a \right).
           $
We have from the triangular inequality, Jensen's inequality and \eqref{25} that
               $$
\mathbb{E} \left( \dfrac{ \max_{\mathbf{1\prec k\prec n}}
\left\|\sum_{\mathbf{i\prec k}} (X_{\mathbf{i}}
-\mathbb{E}X_{\mathbf{i}})
                    \right\| }{|\mathbf{n}|} \right)
                    \le \dfrac{1}{|\mathbf{n}|} \mathbb{E} \left( \max_{\mathbf{1\prec k\prec n}} \left\|
                    \sum_{\mathbf{i\prec k}} \left( Y_{\mathbf{i}} -\mathbb{E}Y_{\mathbf{i}} \right) + \sum_{\mathbf{i\prec k}} \left( Z_{\mathbf{i}} -\mathbb{E}Z_{\mathbf{i}} \right)       
                    \right\| \right)             
$$

$$
\le \dfrac{1}{|\mathbf{n}|} \left[
                    \mathbb{E}\left( \max_{\mathbf{1\prec k\prec n}} \left\|
                    \sum_{\mathbf{i\prec k}} \left( Y_{\mathbf{i}} -\mathbb{E}Y_{\mathbf{i}} \right) \right\| \right) +
                    \mathbb{E} \left( \max_{\mathbf{1\prec k\prec n}} \left\| \sum_{\mathbf{i\prec k}} \left( Z_{\mathbf{i}} -\mathbb{E}Z_{\mathbf{i}} \right) \right\| \right)
                    \right]
                  $$
                  $$
                     \le \dfrac{1}{|\mathbf{n}|}\sum_{\mathbf{i\prec n}} \mathbb{E} \| Y_\mathbf{i} - \mathbb{E} Y_\mathbf{i}\|
                     + \dfrac{1}{|\mathbf{n}|}\left[ \mathbb{E} \max_{\mathbf{1\prec k\prec n}} \left\| \sum_{\mathbf{i\prec k}}
                     \left( Z_{\mathbf{i}} -\mathbb{E}Z_{\mathbf{i}} \right) \right\| ^2 \right]^{1/2}
$$
$$
                      \le \dfrac{2}{|\mathbf{n}|}\sum_{\mathbf{i\prec n}} \mathbb{E} \| Y_\mathbf{i}\|
                    +  \dfrac{1}{|\mathbf{n}|}\left[ \mathbb{E} \max_{\mathbf{1\prec k\prec n}}
                    \left\| \sum_{\mathbf{i\prec k}} \left( Z_{\mathbf{i}} -\mathbb{E}Z_{\mathbf{i}} \right) \right\| ^2 \right]^{1/2}
$$
\begin{equation}\label{26}
                     \le \varepsilon + \dfrac{1}{|\mathbf{n}|}\left[ \mathbb{E} \max_{\mathbf{1\prec k\prec n}}
                     \left\| \sum_{\mathbf{i\prec k}} \left( Z_{\mathbf{i}} -\mathbb{E}Z_{\mathbf{i}} \right) \right\| ^2 \right]^{1/2}.
\end{equation}
            By applying Lemma \ref{lemma24}, there exists a constant $C>0$ such that
$$
               \dfrac{1}{|\mathbf{n}|}\left[ \mathbb{E} \max_{\mathbf{1\prec k\prec n}} \left\| \sum_{\mathbf{i\prec k}}
                    \left( Z_{\mathbf{i}} -\mathbb{E}Z_{\mathbf{i}} \right) \right\| ^2 \right]^{1/2}
                  $$
                  $$
                    \leq \dfrac{C\log(2n_1)\log(2n_2)\cdots \log(2n_d)}{|\mathbf{n}|} \left( \sum_{\mathbf{i\prec n}} E\|Z_{\mathbf{i}} - \mathbb{E}Z_{\mathbf{i}}\|^2 \right)^{1/2}
$$
$$
                    \leq \dfrac{C\log(2n_1)\log(2n_2)\cdots \log(2n_d)}{|\mathbf{n}|} \left( 4\sum_{\mathbf{i\prec n}} E\|Z_{\mathbf{i}}\|^2 \right)^{1/2}
$$
\begin{equation}\label{27}
                     \leq \dfrac{2aC\log(2n_1)\log(2n_2)\cdots \log(2n_d)}{|\mathbf{n}|^{1/2}} \rightarrow 0 \text{ as } |\mathbf{n}| \rightarrow \infty.
\end{equation}
            Since $\varepsilon>0$ is arbitrary, it follows from \eqref{26} and \eqref{27} that
            \[
            \mathbb{E} \left( \dfrac{ \max_{\mathbf{1\prec k\prec n}} \left\|\sum_{\mathbf{i\prec k}} (X_{\mathbf{i}} -\mathbb{E}X_{\mathbf{i}})  \right\| }{|\mathbf{n}|} \right) \rightarrow 0 \text{ as } |\mathbf{n}| \rightarrow 0.
            \]
            Thus \eqref{24} holds.
        \end{proof}

\begin{remark}
    {
        \rm
        \begin{description}
            \item[(i)] In Hong and Volodin \cite{hong1999marcinkiewicz-type},
            the author proved a result on $L_p$ convergence for
            double array of random variables, $0< p<2$ under a stochastic domination condition
            which is much stronger than condition being uniform integrability under the Ces\`{a}ro sense.

            \item[(ii)] By the same method, we can extend Theorem \ref{theorem22} to random
            variables which are coordinatewise and pairwise negative dependent (see \cite{hien2019negative}).
            Ord{\'o}{\~n}ez Cabrera et al. \cite{ordonezcabrera2020new} recently introduce the concept of
            negatively correlated integrability $\mathcal{H}$-random vectors. It would
            be interesting to extend Theorem \ref{theorem22} to this new dependence structure.
        \end{description}
    }
\end{remark}

\section{Some basic results on the Ces\`{a}ro uniform integrablity for $d$-dimensional arrays of random variables}

The concept of uniform integrability in the Ces\`{a}ro sense for sequences of
real valued-random variables was introduced by Chandra \cite{chandra1989uniform}. In
Chandra \cite{chandra1989uniform} and Chandra and Goswami \cite{chandra1992cesaro}, the authors
established criteria for this new type of uniform integrablity for
sequences ($1$-dimension) of random variables.
In this section, we will present these
results for $d$-dimensional arrays for completeness.

We note that a criterion for uniform integrability in the Ces\`{a}ro sense for sequences of random variables
was proved by Chandra \cite[Theorem 3]{chandra1989uniform}. Similar to Theorem 3 of Chandra \cite{chandra1989uniform}, we have the following criterion for uniform integrability in the Ces\`{a}ro sense for $d$-dimensional arrays of random variables.

\begin{theorem}\label{theorem31}
    A $d$-dimensional array of random variables $ \left \{ X_{\mathbf{n}}, \mathbf{n}
    \in \mathbb{Z}^{d}_{+} \right \}$ is uniformly intergrable in the Ces\`{a}ro sense if and only if
    \begin{description}
        \item[(i)] $\sup\limits_{\mathbf{n} \succ \mathbf{1}} \dfrac{1}{\left | \mathbf{n} \right |}  \sum\limits_{\mathbf{i \prec n}} \mathbb{E}\left( |X_{\mathbf{i}}| \right) < \infty$
    \end{description}
    and
    \begin{description}
        \item[(ii)] for each $\varepsilon > 0$, there exists $\delta>0$ such that whenever $\left\{A_{\mathbf{n}}\right\}_{\mathbf{n}\geq \mathbf{1}}$ is a $d$-dimensional array of events satisfying the condition that
        \begin{equation}\label{31}
            \sup\limits_{\mathbf{n} \succ \mathbf{1}} \dfrac{1}{\left | \mathbf{n} \right |}  \sum\limits_{\mathbf{i \prec n}} \mathbb{P}\left( A_{\mathbf{i}} \right) < \delta
        \end{equation}
        we have
        \begin{equation}\label{32}
            \sup\limits_{\mathbf{n} \succ \mathbf{1}} \dfrac{1}{\left | \mathbf{n} \right |}  \sum\limits_{\mathbf{i \prec n}} \mathbb{E}\left( |X_{\mathbf{i}}|\mathbf{1}\left( A_{\mathbf{i}} \right) \right) < \varepsilon.
        \end{equation}
    \end{description}
\end{theorem}

\begin{proof}
    Firstly, we prove the \lq\lq only if\rq\rq\, part.
    Suppose that $ \left \{ X_{\mathbf{n}}, \mathbf{n} \in \mathbb{Z}^{d}_{+} \right \}$ is a $d$-dimensional array of random variables which is uniformly intergrable in the Ces\`{a}ro sense.
    Let $a_0>0$ be such that
        \begin{equation}\label{33}
        \sup\limits_{\mathbf{n} \succ \mathbf{1}} \dfrac{1}{\left | \mathbf{n} \right |}  \sum\limits_{\mathbf{i \prec n}} \mathbb{E}\left( |X_{\mathbf{i}}|\mathbf{1}\left( |X_{\mathbf{i}}| > a_0 \right) \right) \leq 1.
    \end{equation}
        We have
    \[
    \mathbb{E}\left( | X_{\mathbf{i}} | \right) \leq a_0 + \mathbb{E}\left( |X_{\mathbf{i}}|\mathbf{1}\left( |X_{\mathbf{i}}| > a_0 \right) \right), \, \forall \mathbf{i} \succ \mathbf{1}.
    \]
Therefore,
    \[
    \dfrac{1}{\left | \mathbf{n} \right |}  \sum\limits_{\mathbf{i \prec n}} \mathbb{E}\left( |X_{\mathbf{i}}| \right) \le a_0 + \dfrac{1}{\left | \mathbf{n} \right |}  \sum\limits_{\mathbf{i \prec n}} \mathbb{E}\left( |X_{\mathbf{i}}|\mathbf{1}\left( |X_{\mathbf{i}}| > a_0 \right) \right) \le a_0 + 1, \, \forall \mathbf{n} \succ \mathbf{1} \quad (\text{by } \eqref{33}).
    \]
    Thus (i) holds.

Now, we prove (ii). Fix an $\varepsilon > 0$. Let $a_0 > 0$ be such that
\begin{equation}\label{34}
        \sup\limits_{\mathbf{n} \succ \mathbf{1}} \dfrac{1}{\left | \mathbf{n} \right |}  \sum\limits_{\mathbf{i \prec n}} \mathbb{E}\left( |X_{\mathbf{i}}|\mathbf{1}\left( |X_{\mathbf{i}}| > a_0 \right) \right) \leq \dfrac{\varepsilon}{2}.
    \end{equation}
    Put $\delta = \dfrac{\varepsilon}{2a_0}$. If \eqref{31} holds, then
        $$
         \dfrac{1}{\left | \mathbf{n} \right |}  \sum\limits_{\mathbf{i \prec n}} \mathbb{E}\left( |X_{\mathbf{i}}|\mathbf{1}\left( A_{\mathbf{i}} \right) \right)
        \leq \dfrac{1}{\left | \mathbf{n} \right |}  \sum\limits_{\mathbf{i \prec n}}
        \left(
        a_0\mathbb{P}\left( A_{\mathbf{i}} \right) + \mathbb{E}\left( |X_{\mathbf{i}}|\mathbf{1}\left( |X_{\mathbf{i}}| > a_0 \right) \right)
        \right)
        $$
        $$
         = a_0 \, \dfrac{1}{\left | \mathbf{n} \right |}  \sum\limits_{\mathbf{i \prec n}}
         \mathbb{P}\left( A_{\mathbf{i}} \right) + \dfrac{1}{\left | \mathbf{n} \right |}
         \sum\limits_{\mathbf{i \prec n}} \mathbb{E}\left( |X_{\mathbf{i}}|\mathbf{1}\left( |X_{\mathbf{i}}| > a_0 \right) \right)
         \leq a_0\delta + \dfrac{\varepsilon}{2} \quad (\text{by }
         \eqref{34})  = \varepsilon
    $$
thereby proving (ii).

    Secondly, we prove the \lq\lq if\rq\rq\, part. Let $\varepsilon > 0$ be fixed. By (ii), there exists a $\delta > 0$ such that \eqref{31} implies \eqref{32}.
    Put
    \[
    K = \sup\limits_{\mathbf{n} \succ \mathbf{1}} \dfrac{1}{\left | \mathbf{n} \right |}  \sum\limits_{\mathbf{i \prec n}} \mathbb{E}\left( |X_{\mathbf{i}}| \right).
    \]
    By applying Markov's inequality, we have
    \[
    \mathbb{P} \left( | X_{\mathbf{i}} | \geq \dfrac{K}{\delta} \right) \leq \dfrac{\mathbb{E} \left( | X_{\mathbf{i}} | \right) }{ K/\delta }, \, \forall \mathbf{i} \succ \mathbf{1}.
    \]
    Therefore,
    \[
    \dfrac{1}{\left | \mathbf{n} \right |}  \sum\limits_{\mathbf{i \prec n}} \mathbb{P}\left( |X_{\mathbf{i}}| \geq \dfrac{K}{\delta} \right) \le \dfrac{K}{K/\delta} = \delta,\, \forall \mathbf{n} \succ \mathbf{1}.
    \]
        By applying (ii) with $\left\{A_{\mathbf{i}} = \left( | X_{\mathbf{i}} | \geq \dfrac{K}{\delta} \right) \right\}_{\mathbf{i}\geq \mathbf{1}}$, we obtain
        \begin{equation}\label{35}
        \sup\limits_{\mathbf{n} \succ \mathbf{1}} \dfrac{1}{\left | \mathbf{n} \right |}  \sum\limits_{\mathbf{i \prec n}} \mathbb{E}\left( |X_{\mathbf{i}}|\mathbf{1}\left( | X_{\mathbf{i}} | \geq \dfrac{K}{\delta} \right) \right) < \varepsilon.
    \end{equation}
    If $a > \dfrac{K}{\delta}$, then
    \[
    \mathbb{E}\left( |X_{\mathbf{i}}|\mathbf{1}\left( |X_{\mathbf{i}}| > a \right) \right) \leq \mathbb{E}\left( |X_{\mathbf{i}}|\mathbf{1}\left( |X_{\mathbf{i}}| > \dfrac{K}{\delta} \right) \right)
    \]
    which implies that
    \[
    \sup\limits_{\mathbf{n} \succ \mathbf{1}} \dfrac{1}{\left | \mathbf{n} \right |}  \sum\limits_{\mathbf{i \prec n}} \mathbb{E}\left( |X_{\mathbf{i}}|\mathbf{1}\left( |X_{\mathbf{i}}| > a \right) \right) < \varepsilon\, (\text{by } \eqref{35}).
    \]
    Hence the proof is completed.
\end{proof}

Chandra and Goswami \cite[p. 228]{chandra1992cesaro} established
the de La Vall\'{e}e Poussin criterion for uniform integrability in the Ces\`{a}ro sense for sequences of random variables.
This result for $d$-dimensional arrays of random variables can be stated as follows.

\begin{theorem}\label{theorem32}
    Let $ \left \{ X_{\mathbf{n}}, \mathbf{n} \in \mathbb{Z}^{d}_{+} \right \} $ be a $d$-dimensional array of random variables. Consider the statement (A)
    below: there exists a measurable function $\phi$ defined on   $\left[0,\infty\right)$ with $\phi(0) = 0$ such that $\dfrac{\phi (t)}{t} \rightarrow \infty$
    as $t\rightarrow \infty $ and  $\sup_{\mathbf{n} \succ \mathbf{1}} \dfrac{1}{|\mathbf{n}|} \sum_{\mathbf{i\prec n}} \mathbb{E}\left( \phi\left( |X_\mathbf{i}|\right)  \right)<\infty$.
    \begin{description}
        \item[(i)] If (A) holds, then $ \left \{ X_{\mathbf{n}}, \mathbf{n} \in \mathbb{Z}^{d}_{+} \right \} $ is uniformly intergrable in the Ces\`{a}ro sense.
        \item[(ii)] If $\left \{ X_{\mathbf{n}}, \mathbf{n} \in \mathbb{Z}^{d}_{+} \right \} $ is uniformly intergrable in the Ces\`{a}ro sense, then (A) holds. Moreover, $\phi$ can be chosen so that $\dfrac{\phi(t)}{t}$ is increasing.
    \end{description}
\end{theorem}

\begin{proof}
    (i) Assuming that $(A)$ holds. Put
   $
    K = \sup_{\mathbf{n} \succ \mathbf{1}} \dfrac{1}{|\mathbf{n}|} \sum_{\mathbf{i\prec n}} \mathbb{E}\left( \phi\left( |X_\mathbf{i}|\right)  \right).
  $
    For each $\varepsilon > 0$, there exists an integer $N\geq 1$ such that $\phi (t) \geq \dfrac{t(K+1)}{\varepsilon}$ for $t> N$. Then for each $\mathbf{n} \succ \mathbf{1}$ and each $a > N$,
    \[
    \dfrac{1}{|\mathbf{n}|}\sum\limits_{\mathbf{i \prec n}} \mathbb{E}\left( |X_{\mathbf{i}}|\mathbf{1}\left( |X_{\mathbf{i}}| > a \right) \right) \leq
    \dfrac{1}{|\mathbf{n}|}\sum\limits_{\mathbf{i \prec n}} \mathbb{E}\left( \dfrac{\varepsilon \phi\left( |X_{\mathbf{i}}| \right) \mathbf{1}\left( |X_{\mathbf{i}}| > a \right) }{K+1}  \right) <\varepsilon.
    \]
    It follows that
    \[
    \lim_{a\to \infty}  \sup\limits_{\mathbf{n} \succ \mathbf{1}} \dfrac{1}{\left | \mathbf{n} \right |}  \sum\limits_{\mathbf{i \prec n}} \mathbb{E}\left( |X_{\mathbf{i}}|\mathbf{1}\left( |X_{\mathbf{i}}| > a \right) \right) = 0.
    \]
    Thus $ \left \{ X_{\mathbf{n}}, \mathbf{n} \in \mathbb{Z}^{d}_{+} \right \} $ is uniformly intergrable in the Ces\`{a}ro sense.

    (ii) Firstly, we construct the function $\phi(\cdot)$ such that for any increasing sequence $\{u_n,n\geq 1\}$ of nonnegative real numbers
    with $u_n\rightarrow \infty$ as $n\rightarrow \infty$, we put $\phi(0) = 0$ and
\begin{equation}\label{key01}
        \phi(t)=\int_{0}^{t} g(x) \mathrm{d} x, \ t>0,
\end{equation}
where $g\colon (0,\infty) \rightarrow (0,\infty)$ is defined as
\begin{equation}\label{key02}
        g(x)=u_{n} \text { if } n-1 \leq x<n,\ n \geq 1.
\end{equation}

    Secondly, we prove that $f(t) := \dfrac{\phi (t)}{t}$ is increasing and $f(t) \rightarrow \infty $ as $t\to\infty$.
    Fix $n\geq 1$ and $t\in(n-1,n)$. Let $u_0 = 0$. Since $\phi(t) = \sum_{i=0}^{n-1} u_i +
    (t-n+1)u_n$,
   $
    t^2f'(t) = \sum_{i=0}^{n-1}(u_n - u_i) \geq 0.
   $
    It implies that $f'(t)\ge 0$ on each interval $(n-1,n)$, $n\ge1$. Combining this with the fact that $f(t)$ is continuous on $(0,\infty)$ yields $f(t)$ is increasing.
    Since $u_n \rightarrow \infty$ as $n\rightarrow \infty$, $f(n) = \dfrac{1}{n} \sum_{i=1}^{n} u_i \rightarrow \infty$ as $n\rightarrow \infty$. Therefore, $f(t) \rightarrow \infty $ as $t\rightarrow\infty$.

    Finally, we prove that such  sequence $\{u_n,n\geq 1\}$ can be chosen so as to also satisfy
    \begin{equation}\label{36}
        \sup_{\mathbf{n} \succ \mathbf{1}} \dfrac{1}{|\mathbf{n}|} \sum_{\mathbf{i\prec n}} \mathbb{E}\left( \phi\left( |X_\mathbf{i}|\right)  \right) \leq 1.
    \end{equation}
    Suppose that $\left \{ X_{\mathbf{n}}, \mathbf{n} \in \mathbb{Z}^{d}_{+} \right \} $ is uniformly intergrable in the Ces\`{a}ro sense, then there exists a sequence $\{N_j,j\geq 1\}$ of positive integers such that $N_j \rightarrow \infty$ as $j\rightarrow \infty$ and for each $j\geq 1$,
    \begin{equation}\label{37}
        \sup_{\mathbf{n} \succ \mathbf{1}} \dfrac{1}{|\mathbf{n}|} \sum_{\mathbf{i\prec n}} \mathbb{E}\left( |X_{\mathbf{i}}| \mathbf{1}\left( |X_\mathbf{i}| \geq N_j \right)  \right) \leq 2^{-j}.
    \end{equation}
    For each $n\ge 1$, we consider a sequence $\{u_n,n\geq 1\}$ defined by $u_n={\rm card}(A)$, where
\[   A=\{j\in \mathbb{N}: N_j<n\}.\]
    Clearly, $\{u_n,n\ge1\}$ is increasing, and so that $\phi(t)$ is increasing.
    Since $N_j \rightarrow \infty$ as $j\rightarrow \infty$, we have $u_n\rightarrow \infty$ as $n\rightarrow \infty$.
    From \eqref{key01} and \eqref{key02}, we have for $t\in (n-1;n)$ and $u_0 = 0$ that
    \[\phi(t) = \sum_{i=0}^{n-1} u_i + (t-n+1)u_n \leq \sum_{i=1}^{n} u_i.\]
    Therefore,
   $$
         \mathbb{E}\left( \phi\left( |X_\mathbf{k}|\right)  \right)
        \leq \mathbb{E} \left[ \sum_{m=1}^{\infty} \left( \mathbf{1}\left( m - 1 < |X_{\mathbf{k}}| \leq m \right) \sum_{i=1}^{m} u_i \right)  \right]
$$
$$
  = \mathbb{E}\left[ \sum_{i=1}^{\infty} u_i \left( \sum_{m=i}^{\infty} \mathbf{1}\left( m - 1 < |X_{\mathbf{k}}| \leq m \right) \right) \right]
  = \sum_{i=1}^{\infty} u_i \mathbb{P} \left( |X_{\mathbf{k}}| > i - 1 \right),\ \mathbf{k}\succ \mathbf{1}.
   $$
    It follows that for any $\mathbf{n} \succ \mathbf{1}$,
   $$
         \sum_{\mathbf{1\prec k\prec n}} \mathbb{E}\left( \phi\left( |X_\mathbf{k}|\right)  \right)
        \leq \sum_{\mathbf{1\prec k\prec n}} \sum_{i=1}^{\infty} u_i \mathbb{P} \left( |X_{\mathbf{k}}| > i - 1 \right)
$$
$$
   \leq \sum_{\mathbf{1\prec k\prec n}} \sum_{i=1}^{\infty} \left[ \left( \sum_{j\colon N_j < i} 1 \right) \mathbb{P} \left( |X_{\mathbf{k}}| > i - 1 \right) \right]
        \leq \sum_{j=1}^{\infty}  \sum_{\mathbf{1\prec k\prec n}} \sum_{i = N_j}^{\infty} \mathbb{P} \left( |X_{\mathbf{k}}| > i \right)
$$
$$
  =\sum_{j=1}^{\infty}  \sum_{\mathbf{1\prec k\prec n}} \mathbb{E}\left( |X_{\mathbf{k}}| \mathbf{1}\left( |X_\mathbf{k}| \geq N_j \right)  \right)
         \leq |\mathbf{n}|\sum_{j=1}^{\infty} 2^{-j} \quad (\text{by } \eqref{37})  = |\mathbf{n}|.
    $$
    This implies \eqref{36} holds. The proof of the theorem is completed.
\end{proof}

\end{document}